\documentclass[11pt]{amsart}
\usepackage{a4wide}
\usepackage[latin1, utf8]{inputenc}

\usepackage{amssymb}
\usepackage{amsmath}
\usepackage{amsthm} 

\usepackage{mathtools} 
\mathtoolsset{showonlyrefs}

\usepackage{graphicx}
\usepackage{caption}
\usepackage{framed}

\usepackage{amsfonts}

\usepackage{hyperref}
\usepackage{xcolor}

\hypersetup{
	colorlinks   = true, 
	urlcolor     = blue, 
	linkcolor    = blue, 
	citecolor   = blue 
}

\numberwithin{equation}{section}
\newtheorem{theorem}{Theorem}[section]
\newtheorem{lemma}[theorem]{Lemma}
\newtheorem{claim}[theorem]{Claim}

\newtheorem{remark}[theorem]{Remark}
\newtheorem{proposition}[theorem]{Proposition}

\newtheorem{corollary}[theorem]{Corollary}

\newcommand{\R}{{\mathbb R}}

\newcommand{\Z}{{\mathbb Z}}
\newcommand{\N}{{\mathbb N}}
\newcommand{\G}{\mathcal{G}}

\allowdisplaybreaks
\DeclareMathOperator*{\Var}{Var\,}

\title[Asymptotic safety regions for Hermite functions]{Asymptotic safety regions for Gabor frames generated by Hermite functions}

\author[Faulhuber]{Markus Faulhuber}
\address{Faculty of Mathematics, University of Vienna \\ Oskar-Morgenstern-Platz 1, 1090 Vienna, Austria}
\email{markus.faulhuber@univie.ac.at}

\author[Shafkulovska]{Irina Shafkulovska}
\address{Faculty of Mathematics, University of Vienna \\ Oskar-Morgenstern-Platz 1, 1090 Vienna, Austria}
\email{irina.shafkulovska@univie.ac.at}

\author[Zlotnikov]{Ilya Zlotnikov}
\address{Department of Mathematical Sciences, Norwegian University of Science and Technology (NTNU), 7491 Trondheim, Norway} 
\address{
Erwin Schrödinger International Institute for Mathematics and Physics (ESI),
University of Vienna, Vienna, Austria
}
\address{ Department of Mathematics, King’s College London, Strand, London, WC2R 2LS, United Kingdom} 
\email{ilia.zlotnikov@kcl.ac.uk}

\keywords{Gabor frame, Hermite function, Laguerre polynomial}
\subjclass[2020]{Primary 42C15; Secondary 33C45, 42C05}

\begin{document}

\begin{abstract}

The aim of this paper is to establish new regions in the frame sets of Hermite functions $h_n$. 
A classical result of Gr\"ochenig and Lyubarskii shows that the Gabor system
$\mathcal{G}(h_n,a\Z\times b\Z)$ forms a frame for $L^2(\R)$ whenever the lattice density exceeds $n+1$. 
We show that, for every $\eta>0$ and all sufficiently large $n$, the same Gabor system forms a frame whenever
$ab\leq n^{-\frac{2}{3}-\eta}$.

Moreover, we obtain an asymptotically sharp result near the coordinate
axes, i.e., when one of the parameters $a$ or $b$ is small. Namely, for every $\delta>0$ and $\rho\in(0,\frac{1}{2})$ and all sufficiently large $n$ we prove that if $\min\{a,b\}\leq n^{-\frac{1}{2}-\delta}$ and
$ab\leq \frac{1}{2}-\rho$ then $\mathcal{G}(h_n,a\Z\times b\Z)$ forms a frame.

\end{abstract}

\date{}
\maketitle

\section{Introduction and Main Results}
\subsection{Introduction}
For a window $g\in L^2(\R)$ and $a,b>0$ the Gabor system $\mathcal{G}(g, a\Z \times b\Z)$ consists of time-frequency shifts of $g$ and is defined by
$$
\mathcal{G}(g, a, b):=\mathcal{G}(g, a\Z \times b\Z) := \{g_{k, l}(t) = e^{2 \pi i b l t} g(t - ak) : k \in \Z, l \in \Z\}.
$$
A fundamental problem in time-frequency analysis is to determine the assumptions to be imposed on $g,a,b$ to ensure that the system $\mathcal{G}(g, a\Z \times b\Z)$ 
forms a frame for $L^2(\R)$, that is, there exist two constants $A, B > 0$ such that for all $f\in L^2(\R)$ the inequalities
$$A\|f\|_2^2 \le\sum_{(k, l)\in \Z^2}|\langle f, g_{k,l}\rangle|^2 \le B\|f\|_2^2$$
are true. In other words, we ask for the description of the frame set 
\begin{equation}
    \mathcal{F}(g):= \left\{(a,b) \in \R^2_+ : \mathcal{G}(g,a,b) \text{ forms a frame for } L^2(\R) \right\}.
\end{equation}

A complete solution to this problem is known only for a few individual windows
$g$, including the Gaussian kernel \cite{MR1188007, MR1173118}, the hyperbolic secant \cite{MR1884237}, the symmetric exponential function \cite{MR1964306}, the Cauchy kernel \cite{MR4345947}, and the indicator function of an interval \cite{MR3545108}. More generally, complete descriptions of the frame set are also known for certain families of windows,
such as totally positive functions \cite{totallypos_full2026, MR3763405, MR3053565}, rational functions with poles in
the same half-plane \cite{MR4542702}, and one-sided exponentially decaying functions \cite{MR5077907}.

For other important classes of functions like polynomials with Gaussian weight, $B$-splines, and ratios of exponential polynomials, only partial descriptions
of the frame set are available, see \cite{MR4484792, MR3398948, MR2529475, MR4832036, MR4865221}. 

The Hermite functions are at the heart of Gabor analysis, and the structure of their frame sets has been investigated in several papers. Following \cite{MR2529475}, we define the $n$-th Hermite function by
\begin{equation}\label{eq:hn}
    h_n(t)
    =
    c_n(-1)^n e^{\pi t^2}
    \frac{d^n}{dt^n}e^{-2\pi t^2},    
    \qquad t\in\R, \ n \in \N_0,
\end{equation}
where the positive constant $c_n$ is chosen so that $\|h_n\|_2=1$ and set
\begin{equation}
    \G(h_n,a\Z\times b\Z)=:\G_n(a,b),
    \qquad a,b>0.
\end{equation}

The first natural example is the function $h_0$, i.e., the Gaussian kernel,
and the frame properties of the Gabor system $\mathcal{G}(h_0,a,b)$ were completely described by Seip and Lyubarskii in papers \cite{MR1188007, MR1173118}:
\begin{equation}
    \mathcal{F}(h_0) = \left\{(a,b) \in \R^2_+ : ab < 1 \right\}.
\end{equation}

For higher-order Hermite functions, the structure of the frame set is significantly more complicated (see the comprehensive review \cite{MR3232589}).
The known results on the structure of $\mathcal{F}(h_n)$ can be divided into two groups. 

The statements from the first group describe the pairs $(a,b)$ that do not belong to the frame set $\mathcal{F}(h_n)$.
For instance, in \cite{MR3027914}, Lyubarskii and Nes proved that no point on any of the hyperbolas $ab=\frac{m}{m+1}, m \in \N$ belongs to the frame set $\mathcal{F}(h_{2n-1}), n \in \N.$ 
Further obstructions to the frame property were obtained by
Lemvig~\cite{MR3624951} and substantially extended by Horst, Lemvig, and
Videb{\ae}k~\cite{MR4849800}. Together, these results provide non-frame
points for every Hermite function $h_n$, $n\geq2$. The constructions
in~\cite{MR4849800} include several rational hyperbolas, namely
$ab=\frac12,\frac13,\frac14,$ and $\frac23$. Moreover, the location of
their non-frame points depends on $n$: one lattice parameter is of order
$n^{1/2}$, while the other is of order $n^{-1/2}$.

The second group consists of results that provide sufficient conditions on the parameters $(a,b)$ ensuring that the Gabor system $\mathcal{G}_n(a,b)$ forms a frame for $L^2(\R)$. Gr\"ochenig and Lyubarskii in papers \cite{MR2292280, MR2529475} established that
\begin{equation}
    \mathbb{H}_{n+1} := \left\{(a,b) \in \R^2_+ : ab < \frac{1}{n+1} \right\}
    \subset \mathcal{F}(h_n).
\end{equation}
Their proof uses the methods from complex analysis and the connection between the frame properties of Gabor systems and sampling and interpolation properties in Fock spaces. Recently, in \cite{MR4887696}, we showed that at least for the square lattices, i.e., in the case $a=b$, the frame set $\mathcal{F}(h_n)$ can be enlarged beyond $ab<(n+1)^{-1}.$ 
Thus, as $n\to\infty$, there remains a substantial gap between the
general sufficient condition $ab<(n+1)^{-1}$ and the known obstructions
to the frame property.

This paper contributes to this second group of results. We significantly enlarge the region that is known to belong to $\mathcal{F}(h_n).$ Moreover, some regions that we determine to belong to the frame set are asymptotically sharp. 

Our approach exploits several special properties of Hermite functions,
in particular their invariance under the Fourier transform and their
connection with Laguerre polynomials. We use two different frame criteria
in complementary regions of the parameter space: a criterion based on
Janssen's representation when neither lattice parameter is too small,
and a criterion based on a Wirtinger inequality close to the coordinate
axes. We recall the required properties of Hermite and Laguerre functions in Section~2.

\subsection{Main Results}

The following statement is the main result of this paper.

\begin{theorem}\label{thm:main}
    For every $\eta,\delta>0$ and every $\rho\in(0,\frac12)$, there exists $n_{\eta,\delta,\rho}\in\N$ such that, for every $n\geq n_{\eta,\delta,\rho}$, the region
    \begin{equation}\label{eq:main_safety_region}
        \left\{
            (a,b)\in\R_+^2:
            ab\leq n^{-\frac23-\eta}
        \right\}
        \cup
        \left\{
            (a,b)\in\R_+^2:
            ab\leq\frac12-\rho,\quad
            \min\{a,b\}\leq n^{-\frac12-\delta}
        \right\}
    \end{equation}
    is contained in the frame set of the $n$-th Hermite function.
\end{theorem}
\begin{figure}[ht]
    \centering
    \includegraphics[width=0.625\linewidth]{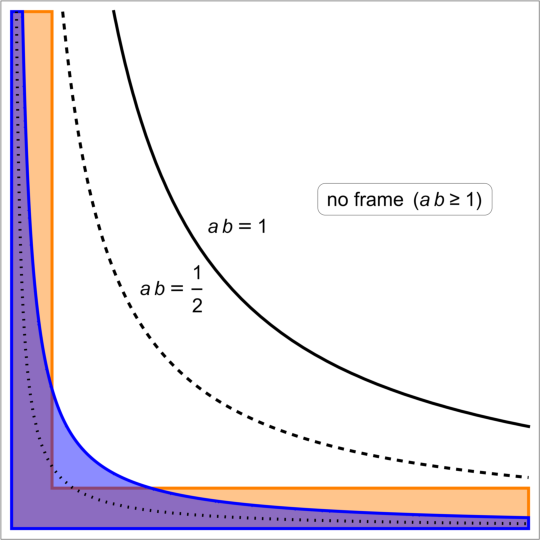}
    \caption{Illustration of the new safety region from Theorem~\ref{thm:main}. The dotted line inside the new safety region represents the hyperbola $ab=1/(n+1)$.}
    \label{fig:new_safety}
\end{figure}

\begin{remark}
The inclusion
\begin{equation}
    \left\{(a,b)\in\R_+^2:
    ab\leq\frac12-\rho,\quad
    \min\{a,b\}\leq n^{-\frac12-\delta}
    \right\}
    \subset \mathcal{F}(h_n)
\end{equation}
provided by Theorem~\ref{thm:main} is essentially asymptotically optimal in two directions.
First, by \cite{MR3027914}, the Gabor system $\mathcal{G}_n(a,b)$ does not form a frame whenever
$ab=\frac12$ and $n$ is odd. Since $\rho>0$ can be chosen arbitrarily small, the threshold $\frac12$ for the product $ab$ cannot be improved in general.
Second, as mentioned above, Horst, Lemvig, and Videb{\ae}k~\cite[Section~5.3]{MR4849800}
construct non-frame points for which one lattice parameter is of order $n^{1/2}$ and the other is of order $n^{-1/2}$. Thus, the exponent $\frac{1}{2}$ in the condition on $\min\{a,b\}$ is also asymptotically optimal.
\end{remark}

\begin{remark}
In fact, Gr\"ochenig and Lyubarskii proved in \cite{MR2529475} a more general
result, which we mentioned in the introduction. Namely, if
$g(x)=p(x)e^{-\gamma x^2}, \gamma>0,$ where $p$ is a polynomial of degree $n$, then
\begin{equation}
    \mathbb{H}_{n+1}\subset\mathcal{F}(g).
\end{equation}
Recently, the third author together with Ulanovskii ~\cite{ulanovskii26} showed that this result
is sharp within the class of polynomials with Gaussian weight. More precisely,
for every $a,b>0$ satisfying $ab=(n+1)^{-1}$ and every $\gamma>0$, there
exists a polynomial $p_{a,b,\gamma}$ of degree $n$ such that, for
$$g(x)=p_{a,b,\gamma}(x)e^{-\gamma x^2},$$
the Gabor system $\mathcal{G}(g,a,b)$ does not form a frame for $L^2(\R)$.
Thus, for a general polynomial multiplied by a Gaussian weight, the region
$\mathbb{H}_{n+1}$ cannot be enlarged uniformly.
\end{remark}

\subsection{Outline of the paper}
The paper is organized as follows. In Section~2, we recall several useful criteria that ensure Gabor systems form frames, introduce the classical Laguerre polynomials, and discuss some of their basic properties, including their connection to Hermite functions and known upper bounds in different regions. In the same section, we also formulate and prove two auxiliary statements concerning estimates for series with rapidly decaying terms and bounds on the number of lattice points in certain regions of the plane. In Section~3, we obtain a new region in the frame set in the regime where neither parameter $a$ nor $b$ is small. In Section~4, we study the complementary regime in which at least one of the parameters $a$ or $b$ is small. We conclude the paper with Section~5, where we prove Theorem~\ref{thm:main}.

\subsection{AI disclosure}
The original approach, developed by the authors in the middle of 2024, yielded an asymptotic safety region of the form $ab \lesssim n^{-5/6}$. Later, during the preparation of the manuscript, the authors used OpenAI's ChatGPT to refine the argument, leading to the improved region $ab\leq n^{-2/3-\eta}$ for every $\eta>0$. ChatGPT was also used to improve the clarity and readability of the exposition. All mathematical arguments were independently verified by the authors, who take full responsibility for the content of the manuscript.

\section{Preliminaries and Notation}
\subsection{Notation}
Below, by $\lceil s\rceil$ we denote the smallest integer greater than or equal to $s$, while $\lfloor s\rfloor$ denotes the largest integer less than or equal to $s$. By $\N$ we denote the set of positive integers and $\N_0 = \N \cup \{0\}$.

The Hilbert space of complex-valued, square-integrable functions on the real line is denoted by $L^2(\R)$. For the norm of $f \in L^2(\R)$ we write $\|f\|_2$, with the usual convention that functions that agree almost everywhere are identified.

The Fourier transform of a function $f$ is given by
\begin{equation}
    \mathcal{F} f (y) = \widehat{f}(y) = \int_{\R} f(t) e^{-2 \pi i y t} \, dt.
\end{equation}
This formula is valid for suitable dense subspaces of $L^2(\R)$, such as the Schwartz space $\mathcal{S}(\R)$. The Fourier transform $\mathcal{F}$ extends to a unitary operator on $L^2(\R)$ by density, and we denote the image of $f$ under the Fourier transform by $\widehat{f}$. This is the usual convention in time-frequency analysis, see, e.g., \cite{Gro01}. It is compatible with the normalization of the Hermite functions we have chosen. In particular, for $h_n$ defined in \eqref{eq:hn}, we have
\begin{equation}
    \mathcal{F} h_n = \widehat{h}_n = (-i)^n h_n, \quad n \in \N_0.
\end{equation}
We refer to \cite{Fol89} for more details.
\subsection{Janssen and Wirtinger tests}

In what follows, our main goal is to determine new regions in the frame sets of Hermite functions. To this end, we will use the following two sufficiency criteria for the Gabor systems $\mathcal{G}_n(a,b)$ to form a frame in $L^2(\R)$.

The first one is based on the Janssen representation of the frame operator. It was introduced in \cite{Tsc00_Master} for finite Gabor systems and in \cite{Wie13_PhD} for the continuous case; see also \cite{MR4887696}.
\begin{proposition}[Janssen test]\label{pro:Janssen} Let $g\in L^2(\mathbb R)$, $\|g\|_2=1,$ and let $\Lambda \subset \R^2$ be a lattice.
    The Gabor system $\G(g,\Lambda)$ is a frame for $L^2(\R)$ if
    \begin{equation}
        \sum_{\lambda^\circ \in \Lambda^\circ} |V_g g(\lambda^\circ)| < 2,
    \end{equation}
    where the notation $\Lambda^\circ$ stands for the adjoint lattice.
\end{proposition}
Note that for $\Lambda=a \Z \times b \Z$ we have $\Lambda^\circ = \frac{1}{b} \Z \times \frac{1}{a} \Z$.

The next sufficiency criterion is based on a Wirtinger inequality and was used in \cite{GosSle23, MR3707488, SunZho03}.
\begin{proposition}[Wirtinger test]\label{pro:Wirtinger}
Let $g\in L^2(\mathbb R)$ be nonzero. Define
\begin{equation}\label{eq:Wirtinger}
    \delta_g:= \frac{1}{2} \inf_{x\in[0,1]}
    \sqrt{\frac{\displaystyle\sum_{k\in\mathbb Z}|\widehat g((x-k))|^2}{\displaystyle\sum_{k\in\mathbb Z}(x-k)^2|\widehat g((x-k))|^2}}.
\end{equation}
Then, the Gabor system $\mathcal{G}(g,\delta \mathbb Z\times \mathbb Z)$ is a frame for $L^2(\mathbb R)$ for all $\delta < \delta_g$.
\end{proposition}

The (unitary) dilation operator $\mathcal{D}_L$, $L > 0$, acts on a function by the rule
\begin{equation}
    \mathcal{D}_L f(t) = \sqrt{1/L} \ f(t/L).
\end{equation}
There is a well-known equivalence between dilated Gabor systems, which is a special case of a more general unitary equivalence principle, see \cite{Fau25_SampTA, Gos15} and \cite[Chap.~9.4]{Gro01}.
\begin{proposition}[Dilation of Gabor frames]\label{pro:dilation}
    The Gabor system $\G(g, a \Z \times b \Z)$ is a frame if and only if $\G(\mathcal{D}_L g, (a L) \Z \times (b/L) \Z)$ is a frame and they have the same frame bounds.
\end{proposition}

Combining Proposition~\ref{pro:Wirtinger} with Proposition~\ref{pro:dilation}, we obtain the following result (cf.~\cite{Fau26}).
\begin{corollary}\label{cor:Wirtinger_rectangular}
    Let $g \in L^2(\R)$ be nonzero. Define
    \begin{equation}\label{eq:Wirtinger_rectangular}
        \delta_g(b) := \delta_{\mathcal{D}_b g}
        = \frac{1}{2} \inf_{x \in [0,1]}
        \sqrt{\frac{\displaystyle\sum_{k\in\mathbb Z}|\widehat g(b(x-k))|^2}{\displaystyle\sum_{k\in\mathbb Z}(x-k)^2|\widehat g(b(x-k))|^2}}
    \end{equation}
    If $a b < \delta_g(b)$, then the Gabor system $\G(g, a \Z \times b \Z)$ is a frame.
\end{corollary}
\begin{proof}
    For the proof, it is necessary to note that the intertwining relation of the dilation operator $\mathcal{D}_b$ and the Fourier transform $\mathcal{F}$ as a unitary operator on $L^2(\R)$ is given by
    \begin{equation}
        \mathcal{F} \mathcal{D}_b = \mathcal{D}_{1/b} \mathcal{F}.
    \end{equation}
    We use the dilation equivalence of Gabor systems in Proposition~\ref{pro:dilation} to transfer the properties of $\G(g, a \Z \times b \Z)$ to the system $\G(\mathcal{D}_b g, \delta\Z \times \Z)$, $\delta=a b$. By Proposition~\ref{pro:Wirtinger}, this system is a frame whenever $\delta < \delta_{\mathcal{D}_b g}$. As $\mathcal{D}_{1/b} \widehat{g}(x-k) = \sqrt{b} \, \widehat{g}(b(x-k))$, the claim follows.
\end{proof}

\subsection{Hermite functions and Laguerre polynomials}
In what follows, our arguments heavily rely on the properties of (generalized) Laguerre polynomials:
\begin{equation}
    \mathcal{L}^{\alpha}_n(x) = \sum_{k=0}^n (-1)^k \binom{n + \alpha}{n - k} \frac{x^k}{k!}.
\end{equation}
Below, we will deal with Laguerre polynomials with $\alpha\in \{ -1/{2},0,1/{2}\}.$

They are related to the Hermite functions via the short-time Fourier transforms (see \cite{Fol89})
\begin{equation}\label{eq:STFT_hermite_laguerre}
    V_{h_n}h_n(x,\xi)
    =
    e^{-\pi i x\xi}
    e^{-\frac{\pi}{2}(x^2+\xi^2)}
    \mathcal{L}_n^0\left(\pi(x^2+\xi^2)\right),
    \qquad (x,\xi)\in\R^2.
\end{equation}
as well as the following identity \cite[eq.~22.5.38, eq.~22.5.39]{AbramowitzStegun1964} that depends on the parity of $n$;
\begin{equation}\label{eq:h_n-L_n}
    h_n(t) = C_n e^{-\pi t^2} 
    \begin{cases}
        \mathcal{L}_{n/2}^{-1/2}(2\pi t^2), &\quad n\in 2\N_0,\\
        \sqrt{2\pi} t \mathcal{L}_{(n-1)/2}^{1/2}(2\pi t^2), &\quad  n\in2\N_0+1,
    \end{cases}
\end{equation}
where $C_n$ depends only on $n$.

\subsection{Upper bounds on Laguerre polynomials}\label{sec:upper_bounds_for_laguerre}

The upper bounds for Laguerre polynomials have been intensively investigated by many researchers; see \cite{MR2168916,MR1662719,MR1424132,Sze39}. A classical estimate for Laguerre polynomials is the \textit{Szeg\H{o} bound} \cite[eq.~(7.21.3)]{Sze39}, which gives for all $n \in \N$ and $\alpha=0$
\begin{equation}\label{eq:Szego}
    |\mathcal{L}^0_n(x)| \leq e^{x/2}, \quad x \geq 0.
\end{equation}
In our proofs, we need more accurate bounds for $\mathcal{L}_n^{\alpha}(x)$ that depend on the position of the point $x$. It is well-known, see e.g. \cite{MR1036513}, that the roots of the Laguerre polynomial $\mathcal{L}^{\alpha}_n$ are contained in the interval $[q, s]$, where 
\begin{equation}
    q(\alpha) := \left(\sqrt{n + \alpha + 1} - \sqrt{n}\right)^2, \quad s(\alpha) := \left(\sqrt{n + \alpha + 1} + \sqrt{n}\right)^2.
\end{equation}
Observe that, for any $n \in \N$ and $\alpha\in\{-1/2,0,1/2\}$,
\begin{equation}
    s(\alpha) = 2n+1+\alpha+2\sqrt{n(n+\alpha+1)}\leq 4n+2\alpha+2, 
    \quad \text{ and } \quad
    q(\alpha)=\frac{(\alpha+1)^2}{s} \leq \frac{(\alpha+1)^2}{4n}.
\end{equation}

For $\alpha = 0$ we set $q:=q(0)$ and $s:=s(0).$
Note that
\begin{equation}\label{eq:q+s=nu}
    q + s = (\sqrt{n+1}-\sqrt{n})^2 + (\sqrt{n+1}+\sqrt{n})^2 = 4n+2 =: \nu.
\end{equation}
Moreover, if $\alpha = 0$, we note the following estimates
\begin{equation}
    s\leq 4n+2 = \nu,
    \quad \text{ and } \quad
    q\leq \frac{1}{4n},
\end{equation}
which we will use in the sequel. The interval $[q,s]$ is typically referred to as the oscillatory region for the $n$-th Laguerre polynomial. Sometimes, this refers to the slightly larger region $[q, \nu]\supseteq [q,s]$, $\nu=4n+2$. The analysis and estimates of $|\mathcal{L}^{0}_n(x)|e^{-x/2}$ are thus often split into four areas, depicted in Figure~\ref{fig:Laguerre_regions}, which we refer to as
\begin{enumerate}
    \item the region \textit{near the origin},
    \item the \textit{oscillatory region} $[q,s]$,
    \item the \textit{transition region} around $s$,
    \item the \textit{asymptotic region}, which starts after $s$, thus after the last root.
\end{enumerate}
We also refer to, e.g., \cite{HighTransFcts1981, Tricomi1949} at this point.

\begin{figure}[h!tp]
    \begin{framed}
    \centering
    \includegraphics[width=0.975\linewidth]{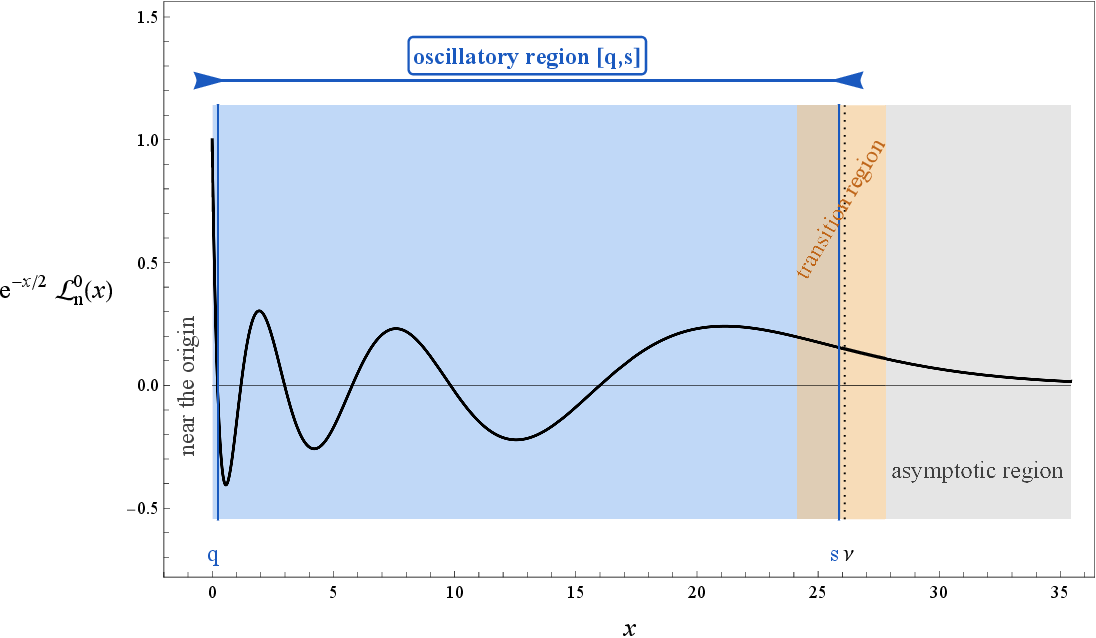}
    \caption{The four different regions for the sixth Laguerre polynomial. For the illustration, we use the function $\mathcal{L}^0_6(x) e^{-x/2}$. For $n=6$, the values of $q$ and $s$ are approximately $q=0.0385$ and $s=25.9615$, and $\nu = 26$. All roots are contained in the oscillatory region.}
    \label{fig:Laguerre_regions}
    \end{framed}
\end{figure}

In Section~3, we restrict to the case $\alpha=0$ and do not evaluate the Laguerre polynomial on $(0,q)$, see \eqref{eq:min_point}, so we do not require an estimate near the origin.
Our argument relies on the following list of estimates for $x\geq q$:
\begin{equation}\label{eq:estimate}
    |\mathcal{L}^{0}_n(x)|e^{-x/2}\leq \begin{cases}
        \sqrt{\frac{s-q}{(s-x)(x-q)}}, &\quad  q\leq x\leq s\\
        C n^{-1/3} &\quad  \nu -\big(\tfrac{4\nu}{3}\big)^{1/3} \leq x\leq \nu +\big(\tfrac{4\nu}{3}\big)^{1/3} \\
        C x^{-1/2} n^{1/6} e^{-R(s,x)/2}, &\quad  x\geq s.
    \end{cases}
\end{equation}
Note that the cases are not distinct. Throughout the manuscript, $C > 0$ will denote a generic constant, independent of $x$ and $n$. Moreover, we denote by $R(s,x)$  the following integral
\begin{equation}
    R(s, x) = \int\limits_{s}^{x} \frac{\sqrt{(q-t)(s-t)}}{t} dt.
\end{equation}
The first estimate can be found in \cite[Thm.~3]{MR2168916}, the second one is \cite[Eq.~40]{Tricomi1949}, see also \cite[Sec.~10.15.]{HighTransFcts1981}, and the last one follows from \cite[Thm.~2]{MR2287374} and \cite[Thm.~5]{MR2287374}.

The estimates in the transition region overlap with the other two:
\begin{equation}
     \nu -\big(\tfrac{4\nu}{3}\big)^{1/3} \leq s \leq \nu +\big(\tfrac{4\nu}{3}\big)^{1/3}
\end{equation}
is equivalent to 
\begin{equation}
    q = \nu - s \leq \big(\tfrac{4\nu}{3}\big)^{1/3}.
\end{equation}
Since $q\leq \tfrac{1}{4n}<1<\nu$, this is true for all $n\in\N$, and provides us more flexibility in our estimates.

Finally, for $\alpha=0, \pm \tfrac{1}{2}$ and sufficiently large $x$, we will use the simple bound:
\begin{equation}\label{eq:x_n_n_factorial_bound}
    |\mathcal{L}^{\alpha}_n(x)| \leq \frac{x^n}{n!}, \quad x>s(\alpha),
\end{equation}
which follows from the product formula $|\mathcal{L}^{\alpha}_n(x)| = \frac{1}{n!}\prod\limits_{k=1}^n|x-x_k|$, where by $x_k$ we denoted the zeros of the Laguerre polynomial.

\subsection{Technical lemma on series estimate}

Below, we will employ the following estimate on the tail of the series with terms of Gaussian decay. This is a variation of a lemma that was kindly suggested to us by the referee of \cite{MR4887696}.
\begin{lemma}\label{lem:tail_estimate}
    Let $ r, M > 0$ and $\Phi:[M,+\infty)$ be differentiable and convex.  If
    \begin{equation}\label{eq:gamma_def}
        \gamma:=\frac{r}{M} -\Phi'(M)<0,
    \end{equation}
    then for every $T \ge M$, the inequality
    \begin{equation}\label{eq:tail_ser_estimate}
        \sum\limits_{l = 0}^{\infty} (T+l)^r e^{-\Phi(T+l)} \leq \frac{M^r e^{-\Phi( M)}}{1-e^{\gamma}}
    \end{equation}
    holds.
\end{lemma}

\begin{proof}
We set
$$
a_l:=(T+l)^{r}e^{-\Phi(T+l)},\quad l\geq 0.
$$
Observe that convexity of $\Phi$ yields
\begin{equation}
    \Phi(T+l+1) - \Phi(T+l) \geq \Phi'(T+l) \geq \Phi'(M).
\end{equation}
Using this inequality and ~\eqref{eq:gamma_def}, for every non-negative integer $l$, we obtain
\begin{equation}
\frac{a_{l+1}}{a_l} = \left(1+\frac{1}{T+l}\right)^{r} 
e^{\Phi(T+l) - \Phi(T+l+1)}
\leq e^{\frac{r}{T+l} - \Phi'(M)} \leq e^{\gamma}.
\end{equation}
This leads to a geometric series estimate
\begin{equation}\label{eq:lemma_tail_prelim}
        \sum\limits_{l=0}^{\infty} (T+l)^r e^{-\Phi(T+l) } = \sum\limits_{l=0}^{\infty} a_l \leq a_0 \sum\limits_{l=0}^{\infty} e^{\gamma l}  \leq \frac{a_0}{1-e^{\gamma}}.
\end{equation}
Moreover, one can check that the convexity of $\Phi$ and condition \eqref{eq:gamma_def} imply that the function $t^r e^{-\Phi( t)}$ is decreasing for $t \geq M$. Hence, $a_0 \leq M^r e^{-\Phi(M)}$. Plugging this into \eqref{eq:lemma_tail_prelim}, we obtain \eqref{eq:tail_ser_estimate}.
\end{proof}

\subsection{Lattice points in ellipses}

In what follows, we will need to estimate the number of lattice points inside certain regions on the plane. This is a rich theory with a number of beautiful results, however, for our purposes we need the following simple lemma.

\begin{lemma}\label{lem:ellipse_lattice_count}
Let $0< a \leq b$ and $r\geq0$. Define an ellipse
    \begin{equation}
        E_r = \left \{(x,y)\in \R^2 :\frac{x^2}{b^2}+\frac{y^2}{a^2}\leq \frac{r}{\pi} \right \}.
    \end{equation}
There exists an absolute constant $C>0$ such that for every $0 \leq r_1 \leq r_2$ we have
\begin{equation}\label{eq:ellipse_difference}
     \#\left(( E_{r_2} \setminus E_{r_1})\cap \Z^2 \right) \leq ab(r_2-r_1) + C(a\sqrt{r_2} + b \sqrt{r_2-r_1} + 1).
\end{equation}
\end{lemma}

\begin{remark}
    In particular, from our lemma, one can deduce
\begin{equation}\label{eq:ellipse}
     \#\left( E_{r}\cap \Z^2 \right) \leq abr + C((a+b)\sqrt{r} + 1).
\end{equation}
We mention that for the lattice points in one ellipse better estimates are available, see e.g.~\cite{Nosarzewska1948}.
\end{remark}

\begin{proof}
Given $r\geq0$ denote by $I_r:=[-a\sqrt{\frac{r}{\pi}}, a\sqrt{\frac{r}{\pi}}]$.
    For  $l \in I_r$ the length of the segment on the level $l$ inside the ellipse $E_{r}$ is given by the function $w_r$ that is $0$ outside $I_r$ and
    \begin{equation}
        w_r(l):= 2b\sqrt{\frac{r}{\pi}-\frac{l^2}{a^2}}, \qquad \text{ for } l\in I_r.
    \end{equation}
Now, we fix $l_0 \in I_{r_2} \cap \Z$ and estimate the number of points from $\Z^2$ inside the ellipse $E_{r_2}$ and outside the ellipse $E_{r_1}$ on the level $l_0$ by
    \begin{equation}
        \#\left\{(k,l_0) \in  ( E_{r_2} \setminus E_{r_1}) \cap \Z^2 \right\} \leq w_{r_2}(l_0) - w_{r_1}(l_0) + 2. 
    \end{equation}
Summing over all admissible integer layers, we get
    \begin{equation}
        \#( (E_{r_2} \setminus E_{r_1}) \cap \Z^2 ) \leq \sum\limits_{l\in \Z \cap I_{r_2}} (w_{r_2}(l) - w_{r_1}(l) + 2). 
    \end{equation}
For brevity, set $\sigma(s):=w_{r_2}(s) - w_{r_1}(s).$
Since $\sigma$ is a function of bounded variation, we have
\begin{equation}
        \sum\limits_{l\in \Z \cap I_{r_2}} (w_{r_2}(l) - w_{r_1}(l)) \leq 
        \int\limits_{I_{r_2}} \sigma(s) ds + \Var\limits_{I_{r_2}} \sigma \leq \mathrm{area}(E_{r_2}) - \mathrm{area}(E_{r_1}) +  \Var\limits_{I_{r_2}} \sigma .
\end{equation}
Using $ \mathrm{area}(E_{r}) = ab r$ and $\#(\Z \cap I_{r_2}) \leq 2 \left(a\sqrt{\frac{r_2}{\pi}} + 1\right)$, we arrive at 
    \begin{equation}\label{eq:E_r_var}
        \#( (E_{r_2} \setminus E_{r_1}) \cap \Z^2 ) \leq ab(r_2-r_1) + 4(a\sqrt{r_2} +1) +  \Var\limits_{I_{r_2}} \sigma. 
    \end{equation}
It remains to estimate the variation of the function $\sigma.$ By symmetry, we may consider only positive half of the $I_{r_2}.$
On the segment $s \in [a\sqrt{r_1/\pi},a\sqrt{r_2/\pi}]$ we see that $\sigma$ is decreasing, since $\sigma(s) = w_{r_2}(s)$. On the other hand, since
$w_{r_2}(s) \geq w_{r_1}(s)$ for all $s \in I_{r_2},$ we have
$$
\sigma'(s) = \frac{4b^2s}{a^2}\left(\frac{1}{w_{r_1}(s)} - \frac{1}{w_{r_2}(s)} \right) \geq 0, \qquad  0\leq s< a\sqrt{r_1/\pi}.
$$
Hence, $\sigma$ is increasing on $[0,a\sqrt{r_1/\pi}].$
Thus,
\begin{equation}
    \Var\limits_{I_{r_2}} \sigma \leq  4\sigma \left(a\sqrt{\frac{r_1}{\pi}} \right) = 4w_{r_2}\left(a\sqrt{\frac{r_1}{\pi}} \right) \leq 8b\sqrt{\frac{r_2-r_1}{\pi}}.
\end{equation}
Together with~\eqref{eq:E_r_var}, this estimate implies~\eqref{eq:ellipse_difference}, which finishes the proof of the lemma.
\end{proof}

\section{Near the diagonal}

In this section, we consider the regime in which neither lattice parameter is too small.
Our main result is the following statement.

\begin{theorem}\label{thm:Janssen_test}
    Let $0< \delta < \eta < \frac{1}{3}$. Then there exists $n_{\eta,\delta}\in\N$ such that, for every $n\geq n_{\eta,\delta}$, the region
    \begin{equation}\label{eq:diagonal_safety_region}
           S_n(\delta,\eta) := \left\{ (a,b) \in \R^2_+ : \min\{a,b\}\geq n^{-\frac12-\delta}, ab \leq n^{-\frac{2}{3}-\eta} \right\} 
        \end{equation}
is contained in the frame set of the $n$-th Hermite function.
\end{theorem}

The proof of this theorem will be achieved by combining two propositions formulated below. 

We start with a brief description of the proof. Our plan is to apply Janssen's test. 
It will be convenient in the sequel to introduce the following object:
\begin{equation}\label{eq:X_k,l}
    X_{k,l}(a,b)
    =\pi\left(\frac{k^2}{b^2}+\frac{l^2}{a^2}\right),
    \qquad a,b>0,\quad (k,l)\in\Z^2.
\end{equation}
Using Proposition~\ref{pro:Janssen} and~\eqref{eq:STFT_hermite_laguerre}, we reduce the problem to verification that, uniformly for $(a,b)\in S_n(\delta,\eta)$, the sum of values of Laguerre polynomials with exponential weights is bounded by $2$ from above:
\begin{equation}\label{eq:Janssen_sum}
    \sum_{(k,l)\in\Z^2}
    \left|V_{h_n}h_n\left(\frac{k}{b},\frac{l}{a}\right)\right|
    =\sum_{(k,l)\in\Z^2}
    \left|\mathcal L^0_n\left(X_{k,l}(a,b)\right)\right|
    e^{-\frac{X_{k,l}(a,b)}{2}}<2
\end{equation}
for all sufficiently large $n$.

It is natural that for sufficiently large indices, depending on the order $n$ of the Laguerre polynomial, the contributions of the corresponding terms to the series become negligible.

We will split the series into a main part, which contains almost all the mass, and a tiny remainder. In a first step, we will determine a threshold value, depending on $n$, beyond which all contributions are sufficiently small, and in fact tend to zero uniformly throughout the search region. Once this is done, we will estimate the elliptic layers of the lattice, which account for the bulk of the mass, in the second step.

Before passing to the proof, we give simple bounds for $X_{k,l}(a,b).$

\begin{remark}
In what follows, by symmetry of the frame set, without loss of generality, we may assume that $a \leq b$. Under this assumption, for $(a,b) \in  S_n(\delta,\eta)$ we have 
\begin{equation}\label{eq:max_bound_diagonal}
   \max\{a,b\} = b  \leq n^{-\frac{2}{3}-\eta}a^{-1}\leq n^{-\frac{1}{6}+\delta-\eta}.
\end{equation}
\end{remark}

Furthermore, from \eqref{eq:max_bound_diagonal} we get
\begin{equation}\label{eq:Xk_lower_bound}
    X_{k,l}(a,b)=\pi\left(\frac{k^2}{b^2}+\frac{l^2}{a^2}\right) \geq  \frac{\pi(k^2 + l^2)}{(\max\{a,b\})^2} \geq \pi (k^2 + l^2) n^{\frac{1}{3}+2(\eta - \delta)}.
\end{equation}
On the other hand, using~\eqref{eq:diagonal_safety_region}, we obtain
\begin{equation}\label{eq:Xk_upper_bound}
    X_{k,l}(a,b)=\pi\left(\frac{k^2}{b^2}+\frac{l^2}{a^2}\right) \leq  \frac{\pi(k^2 + l^2)}{(\min\{a,b\})^2} \leq \pi (k^2 + l^2) n^{1+2\delta}.
\end{equation}

\subsection{Tails in the Janssen test criterion}

We first prove that the tails of the Janssen test for such a Gabor system are asymptotically negligible.

All estimates below are uniform
in $(a,b)\in S_n(\delta,\eta)$; implicit constants may depend on
$\delta$ and $\eta$, but not on $n$, $a$, or $b$.

\begin{proposition}\label{pro:Janssen_tail}
Let $0< \delta < \eta < \frac{1}{3}$, and let $S_n(\delta,\eta)$ be defined by \eqref{eq:diagonal_safety_region}. Then
    \begin{equation}\label{eq:Janssen_tail_limit}
        \lim_{n\to\infty} \sup_{(a,b)\in S_n(\delta, \eta)}
        \sum_{\substack{(k,l)\in\Z^2\\X_{k,l}(a,b)\geq6n}}
        \left|\mathcal L^0_n\left(X_{k,l}(a,b)\right)\right| e^{-X_{k,l}(a,b)/2} = 0.
    \end{equation}
\end{proposition}

\begin{proof}
    
    From the simple bound~\eqref{eq:x_n_n_factorial_bound}, which is valid outside the oscillatory region, and Stirling's inequality $n!\geq\sqrt{2\pi n} \, (n/e)^n$, we get
    \begin{equation}\label{eq:pointwise_tail_bound}
        \left|\mathcal L^0_n(x)\right|e^{-x/2}
        \leq \frac{x^n}{n!}
        \leq \frac{e^{f(n,x)}}{\sqrt{2 \pi n}},
        \qquad x \geq 6n \,(> 4n+2=\nu, \ n \in \N)
    \end{equation}
    where
    \begin{equation}
        f(n,x)=n\bigl(1-\log n+\log x\bigr)-x/2.
    \end{equation}
    The function $f(n,x)$ is decreasing in $x$ on the given range, since
    \begin{equation}
        \partial_x f(n,x)=n/x-1/2<0,
        \qquad x\geq6n.
    \end{equation}
    Moreover,
    \begin{equation}\label{eq:f_6n}
        f(n,6n)=n(1+\log6)-3n=-(2-\log6)n<0.
    \end{equation}

    We seek to estimate, uniformly for
    $(a,b)\in S_n(\delta, \eta)$,
    \begin{align}\label{eq:remainder_Janssen}
        R_n(a,b)
        &: = \frac{1}{\sqrt{2 \pi n}}
        \hspace{-8pt}
            \sum_{\substack{(k,l)\in\Z^2\\X_{k,l}(a,b)\geq6n}}
        \hspace{-8pt} 
        e^{f(n,X_{k,l}(a,b))}\\
        &=\frac{e^{-n(\log n-1)}}{\sqrt{2 \pi n}}
        \hspace{-12pt}
            \sum_{\substack{(k,l)\in\Z^2\\X_{k,l}(a,b)\geq6n}}
        \hspace{-12pt}
        e^{-X_{k,l}(a,b)/2+n\log(X_{k,l}(a,b))}.
    \end{align}
    Set $\varepsilon= \frac{2}{3} -2(\eta-\delta)$. Note that $0<\varepsilon<\frac{2}{3}.$
    The estimates~\eqref{eq:Xk_lower_bound} and ~\eqref{eq:Xk_upper_bound} together imply that for $(a,b)\in S_n(\delta, \eta)$, we have the uniform bounds
    \begin{equation}\label{eq:X_uniform_bounds}
        \pi (k^2 + l^2) n^{\frac{1}{3}+2(\eta - \delta)} = \pi (k^2 + l^2) n^{1-\varepsilon}
        \leq X_{k,l}(a,b)
        \leq \pi (k^2 + l^2) n^{1+2\delta}.
    \end{equation}
    We switch to spherical coordinates and set $m=k^2+l^2$ and    \begin{equation}\label{eq:m0_definition}
        m_0=\left\lceil\frac{6}{\pi}n^\varepsilon\right\rceil.
    \end{equation}
    We use the notation $r_2(m)$ for the number of ways of writing an integer as a sum of two squares, i.e., $r_2(m) = \#\{(k,l) \in \Z^2 \mid k^2+l^2 = m\}$. This is comfortably upper-bounded by
    \begin{equation}\label{eq:r2}
        r_2(m) \leq 4 \sqrt{m}, \quad m \geq 1,
        \qquad
        r_2(0) = 1.
    \end{equation}
    Next, we note that the $m \geq m_0$ suffices to have $X_{k,l}(a,b) \geq 6n$. Yet, $X_{k,l}(a,b) \geq 6n$ does not by itself imply $m\geq m_0$. We therefore first estimate the finite sum of terms for which $m<m_0$. We define
     \begin{equation}
        F_n := F_n(a,b) := \frac{1}{\sqrt{2 \pi n}}
        \hspace{-8pt}
            \sum_{\substack{(k,l)\in\Z^2\\X_{k,l}(a,b)\geq6n\\k^2+l^2<m_0}}
        \hspace{-12pt}
        e^{f(n,X_{k,l}(a,b))}
    \end{equation}
    Using the monotonicity of $f(n,x)$ on $[6n,\infty)$ and the bound \eqref{eq:r2} on $r_2$, we obtain
    \begin{equation}
        F_n \leq  \frac{e^{f(n,6n)}}{\sqrt{2 \pi n}}  \sum_{m=0}^{m_0-1}r_2(m) \leq  \frac{4m_0^{3/2}}{\sqrt{2 \pi n}} e^{f(n,6n)}.
    \end{equation}
    Thus, independently of $(a,b)$ and therefore uniformly, we have
    \begin{equation}
        F_n \leq \frac{4 m_0^{3/2}}{\sqrt{2 \pi n}} e^{-(2-\log(6))n} \to 0, \quad n \to \infty.
    \end{equation}

    It remains to consider $k^2+l^2=m\geq m_0$. In this range, from~\eqref{eq:m0_definition} and ~\eqref{eq:Xk_lower_bound}, we obtain
    \begin{equation}
       X_{k,l}(a,b)\geq \pi m n^{1-\varepsilon} \geq \pi m_0 n^{1-\varepsilon} \geq 6n.
    \end{equation}
    Moreover, since $f(n,x)$ is decreasing in $x$ for $x \geq 6n$, we also have
    \begin{equation}
        f(n,X_{k,l}(a,b)) \leq f(n, \pi n^{1-\varepsilon}m).
    \end{equation}
    Therefore, for $(a,b) \in S_n(\delta, \eta)$ with $X_{k,l} \geq 6n$, we obtain
    \begin{equation}
        R_n \leq F_n + \frac{1}{\sqrt{2\pi n}} \sum_{m=m_0}^{\infty} r_2(m)\, e^{n\big(1-\log n+\log(\pi n^{1-\varepsilon}m)\big)-\tfrac{\pi}{2}\, n^{1-\varepsilon}m}.
    \end{equation}
    We are now in a situation comparable to the estimates carried out in \cite{MR4887696}, where the case $a=b=1/\sqrt{n}$ was studied. Collecting the terms independent of $m$, and bringing them outside of the sum, yields
    \begin{align}
        R_n 
        & \leq F_n +\frac{1}{\sqrt{2\pi n}} \left(\frac{e\pi}{n^{\varepsilon}}\right)^n \sum_{m=m_0}^{\infty} r_2(m)\, m^n\, e^{-\frac{\pi}{2}\,n^{1-\varepsilon} m}\\
        & \leq F_n + \frac{4}{\sqrt{2 \pi n}} \left(\frac{e\pi}{n^{\varepsilon}}\right)^n \sum_{m=m_0}^{\infty} m^{n+\frac{1}{2}} e^{-A m}, \quad A = A_n(\varepsilon) = \frac{\pi}{2} n^{1-\varepsilon}.
    \end{align}
  To bound from above the latter series, we apply Lemma~\ref{lem:tail_estimate} with $M:=m_0, \, r:=n+\frac{1}{2}$ and $\Phi(t):=A \, t.$ It is plain that ~\eqref{eq:gamma_def} is satisfied:
  \begin{equation}
      \gamma:=\frac{r}{M}- \Phi'(M) = \frac{n+\frac{1}{2}}{m_0} -\frac{\pi}{2}n^{1-\varepsilon} \leq -\frac{\pi}{4}n^{1-\varepsilon} < 0.
  \end{equation}
Therefore, we get
\begin{equation}
    \sum_{m=m_0}^{\infty} m^{n+\frac{1}{2}} e^{-A m} \leq \frac{m_0^{n+\frac{1}{2}} e^{-Am_0}}{1 - e^{\gamma}}. 
\end{equation}
By~\eqref{eq:m0_definition}, we have that $Am_0 \geq 3n$. Combining all the estimates together, we arrive at
    \begin{equation}
        R_n \leq F_n + \frac{4}{\sqrt{2 \pi n}} \left( \frac{e \pi}{n^\varepsilon} \right)^n \frac{\left( \frac{6}{\pi} n^\varepsilon + 1 \right)^{n+\frac{1}{2}} e^{-3n}}{1- e^{-\frac{ \pi}{4} n^{1-\varepsilon}}} =: F_n + T_n(\varepsilon).
    \end{equation}
    To finish the proof of the proposition, it suffices to show that for any fixed $0<\varepsilon < \frac{2}{3}$ we have $T_n(\varepsilon) \to 0$ as $n$ tends to $\infty.$ One can easily check that for sufficiently large $n$ we have $e^{-\frac{\pi}{4} n^{1-\varepsilon}}< \frac{1}{2}$, $\pi n^{-\varepsilon}<0.1,$ and therefore 
    \begin{equation}
        T_n(\varepsilon) \leq   (6.1)^{n+\frac{3}{2}} e^{-2n}.
    \end{equation}
    It remains to note that $e^{2}>7$, hence $T_n(\varepsilon) \to 0$ as $n \to \infty.$
\end{proof}

\subsection{The main contributions}
We already know that the remainder sum ($X_{k,l} \geq 6n$) is tiny if $n$ is large enough and asymptotically vanishes. We now need to make sure that the finite sum of the main contributions does not exceed 2. Indeed, we ask for some margin in order to combine the finite sum and the remainder. We will decompose the lattice into elliptic annuli and use contemporary estimates for Laguerre polynomials on each annulus. They are designed to control the Laguerre polynomial in the oscillatory region, with extra emphasis on the region's endpoints. We also gain additional control outside the oscillatory region, where the asymptotics take over. This then allows us to glue together the finite sum and the remainder generously.

Using the notation $X_{k,l}$ introduced in \eqref{eq:X_k,l}, we seek to bound
\begin{equation}
    \sum_{(k,l)\in K} \left|V_{h_n} h_n \left( \frac{k}{b},\frac{l}{a} \right) \right| = \sum_{(k,l)\in K}|\mathcal{L}^{0}_n (X_{k,l}(a,b))|e^{-\frac{X_{k,l}(a,b)}{2}},
\end{equation}
where $K$ collects all points in $\Z^2$ except for the origin in an ellipse of size $10n$.

\begin{proposition}\label{pro:Janssen_main}
    Let $0< \delta < \eta < \frac{1}{3}$, and $S_n(\delta,\eta)$ be defined by \eqref{eq:diagonal_safety_region}. Let $K$ denote the set
    \begin{equation}\label{eq:K_set_def}
        K := K(a,b):= \left\{(k,l)\in\Z^2\setminus\{(0,0)\}:\pi \left( \frac{k^2}{b^2}+\frac{l^2}{a^2} \right) \leq 10n\right\}.
    \end{equation}
    Then
    \begin{equation}
        \lim\limits_{n\to\infty} \sup\limits_{(a,b) \in S_n(\delta, \eta)} \sum_{(k,l)\in K} |\mathcal{L}^{0}_n(X_{k,l}(a,b))| e^{-\frac{X_{k,l}(a,b)}{2}} = 0.
    \end{equation}
\end{proposition}

\begin{proof}
    Without loss of generality, we may assume that $a \leq b$, $(a,b) \in S_n(\delta, \eta)$, such that 
    \begin{equation}\label{eq:X_uniform_bounds_reminder}
        \pi (k^2 + l^2) n^{\frac{1}{3}+2(\eta - \delta)} 
        \leq X_{k,l}(a,b)
        \leq \pi (k^2 + l^2) n^{1+2\delta}.
    \end{equation}
    In particular, this implies that all values $ X_{k,l}(a,b)$ with $(k,l) \in K$ satisfy
    \begin{equation}\label{eq:min_point}
        X_{k,l}(a,b) \geq \pi n^{\frac{1}{3}+2(\eta - \delta)}>\frac{1}{4n}\geq q.
    \end{equation}
    In what follows,  for $X_{k,l}(a,b)$ we will use only these bounds independent of $a$ and $b$.
    
    Next, we fix some parameters $p$ and $t$ satisfying conditions
    \begin{equation}\label{eq:p_t_cond}
        0<t<2/3, \,\, \tfrac{1}{3}<p<1,\quad \text{and} \quad     p+t > \frac{5}{3}-2(\eta - \delta).   
    \end{equation}
    Consider the partition $K = \bigcup_{j=1}^5 K_j$ defined in the following way: 
    \begin{align}
        K_1 &= \{(k,l)\in K: q\leq \pi(\tfrac{k^2}{b^2}+\tfrac{l^2}{a^2})\leq q+n^{t}\}; \\
        K_2 &= \{(k,l)\in K: q+n^{t}< \pi(\tfrac{k^2}{b^2}+\tfrac{l^2}{a^2})\leq s-n^{p}\}; \\
        K_3 &= \{(k,l)\in K: s-n^{p}< \pi(\tfrac{k^2}{b^2}+\tfrac{l^2}{a^2})\leq s-n^{1/3}\}; \\
        K_4 &= \{(k,l)\in K: s-n^{1/3} <\pi(\tfrac{k^2}{b^2}+\tfrac{l^2}{a^2})\leq s\}; \\
        K_5 &= \{(k,l)\in K: s <\pi(\tfrac{k^2}{b^2}+\tfrac{l^2}{a^2})\leq {10}n\}.
    \end{align} 
    Note that, since $0<t,p<1$ and $4n \leq s$, we have $q+n^t<s-n^p$ for all sufficiently large $n\in\N$. Moreover, $p>\frac13$ implies $s-n^p<s-n^{1/3}$ and $s \leq 4n+2 < 10n$, $n \in \N$. Thus, the above partition is well-defined (for $n$ sufficiently large) and the sets $K_j$ are pairwise disjoint.
        
    We now estimate the summands within each subset $K_j, \, j=1, \dots,5$ using the estimates for the Laguerre polynomials described in Subsection~\ref{sec:upper_bounds_for_laguerre}. The layer $K_3$ is more delicate than the others, so we first deal with
    \begin{equation}
        J_1:=\sum_{(k,l)\in K_1\cup K_2 \cup K_4\cup K_5} |\mathcal{L}_n^{0}(X_{k,l}(a,b))|e^{-\frac{X_{k,l}(a,b)}{2}}.
    \end{equation}
    
    For the layers $K_1$ and $K_2$, we use the first estimate in \eqref{eq:estimate} and ~\eqref{eq:X_uniform_bounds}.
    To apply the second estimate in \eqref{eq:estimate} for the layer $K_4$, we need to verify $\nu-(\tfrac{4\nu}{3})^{1/3}<s-n^{1/3}$. Recall from \eqref{eq:q+s=nu} that $\nu =s+q$ and, indeed,
      \begin{equation}
          \nu-(\tfrac{4\nu}{3})^{1/3} \leq s+q-(\tfrac{16n}{3})^{1/3} \leq s - n^{1/3} + \frac{1}{4n} +(1-(\tfrac{16}{3})^{1/3})n^{1/3}
          <s-n^{1/3}
        \end{equation}
    for all $n\geq 3$.
    
    The last envelope in \eqref{eq:estimate}, given by $C x^{-1/2} n^{1/6} e^{-R(s,x)/2}$, where $C$ is a constant independent of $x$ and $n$, is strictly monotonically decreasing in $x$ for $x \geq s$.
    Denoting the number of points in $K_j$ by $\# K_j$ we obtain the following estimate:
    \begin{equation}\label{eq:sum_K123}
        \begin{split}
            J_1
            & \leq \quad \#K_1\cdot \sqrt{\frac{s-q}{(s-(q+n^{t}))(\pi n^{\frac{1}{3}+2(\eta -\delta)}-q)}} \\
            & \quad + \#K_2 \cdot \sqrt{\frac{s-q}{(s-(s-n^{p}))((q+n^{t})-q)}}\\
            & \quad +C \left(\#K_4\cdot n^{-1/3} +\# K_5  \cdot s^{-1/2} n^{1/6} e^{-R(s,s)/2}\right) \\
            &= \#K_1 \cdot \sqrt{\frac{s-q}{(s-q-n^{t})(\pi n^{\frac{1}{3}+2(\eta -\delta)}-q)}} 
             + \#K_2 \cdot \sqrt{\frac{s-q}{n^{p+t}}} \\
            & \quad + C \left(\#K_4\cdot n^{-1/3} + \tfrac{1}{2} \cdot \# K_5 \cdot n^{-1/2+1/6}\right).
        \end{split}
    \end{equation}
    We proceed with simplifying this expression. It holds that $s-q = 4\sqrt{n(n+1)}\in (4n,4n+2)$, and, thus we have
    \begin{equation}
        n^{t}= 4n^{t}-3n^{t} \leq 3(4n-n^{t})
        \leq 3(s-q-n^{t}), 
        \quad 0 < t < 1,
    \end{equation}
    and consequently
    \begin{equation}
        \frac{s-q}{s-q-n^t} = 1+\frac{n^t}{s-q-n^t}\leq 4.
    \end{equation}
    It is clear that for sufficiently large $n$ we have $\pi n^{\frac{1}{3}+2(\eta -\delta)}-q>2n^{\frac{1}{3}+2(\eta -\delta)}$. Thus, we can now further estimate $J_1$ by
    \begin{equation}\label{eq:sum_update}
        \begin{split}
          J_1
            & \leq \#K_1 \cdot \sqrt{\tfrac{4}{2 n^{\frac{1}{3}+2(\eta -\delta)}}} 
            + 3\cdot \#K_2\cdot \sqrt{\tfrac{n}{n^{p+t}}} \\
            & \quad + C \left(\#K_4\cdot n^{-1/3} +\tfrac{1}{2}\# K_5 \cdot n^{-1/3}\right) \\
            &\leq 
            \sqrt{2} \cdot \#K_1 \cdot n^{-\frac{1}{6}-(\eta-\delta)} 
            + 3 \cdot \#K_2\cdot n^{(1-p-t)/2}\\
            & \quad + C \left(\#K_4\cdot n^{-1/3} +\tfrac{1}{2}\# K_5 \cdot n^{-1/3}\right).
        \end{split}
    \end{equation}
    To obtain an upper bound for $J_1$, it remains to estimate $\#K_j$, $j=1,2,4,5$. To this end, we use Lemma~\ref{lem:ellipse_lattice_count}. For $j=2$ and $j=5$ we use~\eqref{eq:ellipse} and the rough estimate  
    \begin{align}\label{eq:K25ab}
        \# K_j
         \leq \#K \leq \#(E_{10n}\cap \Z^2)
         \leq C_j (ab \, n + b \, n^{1/2} + 1),
    \end{align}
    where $C_j$, $j=2,5$, are universal constants, independent of $a, \, b, \, n$.
    Next, the same Lemma, combined with basic estimates, gives us
    \begin{equation}\label{eq:K1ab}
        \# K_1 \leq \#((E_{q+n^t} \setminus E_{q})\cap \Z^2) \leq C_1 ( ab n^{t} + a(q+n^t)^{1/2} + bn^{t/2}+1);
    \end{equation}
    \begin{equation}\label{eq:K4ab}
        \# K_4 \leq \#((E_{s} \setminus E_{s-n^{1/3}})\cap \Z^2)
        \leq C_4 ( ab n^{1/3} + a n^{1/2} + bn^{1/6}+1),
    \end{equation}
    with $C_1$ and $C_4$ again independent of $a, \, b, \, n$.

    Before returning to the estimate of $J_1$, we prefer to eliminate the dependence on $a$ and $b$ in the estimates of the cardinalities of the sets $K_j$. Recall that $0<\delta<\eta<1/3$, $0<t<1$, and 
    \begin{equation}\label{eq:ab_bounds_revisited}
        ab \leq n^{-2/3-\eta}, \quad a \leq \sqrt{ab} \leq n^{-1/3 - \eta/2}, \quad b = \frac{ab}{\min\{a,b\}} \leq \frac{n^{-2/3-\eta}}{n^{-1/2-\delta}} \leq n^{-1/6+\delta - \eta}.  
    \end{equation}
    Combining this with~\eqref{eq:K25ab},\eqref{eq:K1ab}, and \eqref{eq:K4ab}, we arrive at
     \begin{equation}\label{eq:K1245}
         \begin{split}
             \#K_j &\leq C_j(n^{1/3-\eta} +n^{1/3+\delta - \eta} +1)\\
             & \leq C_j n^{1/3+\delta-\eta}, \quad j=2,5;\\
             \#K_1 & \leq C_1(n^{t-2/3-\eta} + n^{t/2 - 1/3 - \eta/2}+n^{t/2-1/6+\delta-\eta}+1)\\
             & \leq C_1(n^{t/2-1/6+\delta-\eta}+1); \\
             \#K_4 & \leq C_4(n^{-1/3-\eta} + n^{1/6 - \eta/2}+n^{\delta-\eta}+1)\\
             & \leq C_4n^{1/6-\eta/2}
             .
         \end{split}
     \end{equation}
   
    Plugging these upper bounds into the estimate of $J_1$ obtained above, we conclude that
    \begin{align}
        &\frac{J_1}{C^*} = \quad \frac{1}{C^*} \sum_{(k,l)\in K_1\cup K_2\cup K_4\cup K_5}  |\mathcal{L}^{0}_n(X_{k,l})| e^{-\frac{X_{k,l}(a,b)}{2}}\\
        & \leq (n^{t/2-1/6+\delta-\eta}+1)  n^{-\frac{1}{6}-(\eta-\delta)}
        + n^{1/3+\delta-\eta} n^{(1-p-t)/2}
        + n^{1/6-\eta/2} n^{-1/3} + \ n^{1/3+\delta-\eta} n^{-1/3}\\
        & \leq n^{t/2-1/3}+n^{-(p+t)/2+5/6+(\delta-\eta)} + n^{-1/6-\eta/2} + n^{\delta-\eta},
    \end{align}
    where  $C^*$ is positive and does not depend on $a, \, b, \, n$. Recall that by choice of $p,t$ (see \eqref{eq:p_t_cond}), we have  $p+t>\tfrac{5}{3}-2(\eta-\delta)$ and $t<\tfrac{2}{3}$. Hence, each of the exponents of $n$ in the last estimate is negative. This implies that
    \begin{equation}
        \lim\limits_{n\to\infty} \sup\limits_{(a,b) \in S_n(\delta, \eta)}\sum_{(k,l)\in K_1\cup K_2\cup K_4\cup K_5}
        |\mathcal{L}^{0}_n(X_{k,l}(a,b))| e^{-\frac{X_{k,l}(a,b)}{2}} = 0.
    \end{equation}

    It remains to prove that the same holds for the sum over $K_3$. Let us rewrite this region as
    \begin{equation}
        K_3 = \{(k,l)\in K: n^{1/3} \leq s - X_{k,l} < n^{p}\}.
    \end{equation}
    We use a dyadic decomposition: define
    \begin{equation}
        M=\lceil \log_2 n^{p-1/3}\rceil , \qquad U_j:=2^j n^{1/3}, \,\,\ j=0,\dots,M,
    \end{equation}
    and split $K_3$ into 
    \begin{equation}
        K_3:=\bigcup\limits_{j=0}^{M-1} V_j := \bigcup\limits_{j=0}^{M-1} \left\{ (k,l) \in K_3: s-X_{k,l} \in [U_j, U_{j+1})  \right\}.
    \end{equation}
    For each layer $j$ and $(k,l)\in V_j$, for sufficiently large $n$, we have
    \begin{equation}
        X_{k,l}-q = s-q-(s-X_{k,l}(a,b)) \geq s-q-n^p \geq n+2, \quad s-X_{k,l}\geq U_j.
    \end{equation}
    Together with the first estimate in~\eqref{eq:estimate}, this gives us
    \begin{equation}\label{eq:Laguerre_on_dyadic_layer}
        |\mathcal{L}_n^{0}(X_{k,l}(a,b))|e^{-\frac{X_{k,l}(a,b)}{2}} \leq  \sqrt{\frac{s-q}{(s-X_{k,l})(X_{k,l}-q)}} \leq \frac{2}{\sqrt{U_j}}.
    \end{equation}

    Next, we estimate the number of lattice points in each layer $V_j$. We use Lemma~\ref{lem:ellipse_lattice_count} and inequality~\eqref{eq:ellipse_difference}:
    \begin{equation}
        \#V_j \leq \#((E_{s-U_j}\setminus E_{s-U_{j+1}})\cap \Z^2) \leq ab (U_{j+1}-U_{j}) + \widetilde{C} \left(a \sqrt{s-U_j} + b \sqrt{U_{j+1}-U_{j}} +1\right),
    \end{equation}
    where $j=0,\dots,M-1$ and $\widetilde{C}$ does not depend on $a, \, b, \, n, \, j$. Since $U_{j+1}-U_j = U_j = 2^j n^{1/3}$ we obtain, using~\eqref{eq:ab_bounds_revisited},
    \begin{equation}
        \#V_j \leq 2^j n^{-1/3-\eta} + \widetilde{C} (n^{1/6-\eta/2} + 2^{j/2} n^{\delta-\eta} +1).
    \end{equation}
    Combining this with~\eqref{eq:Laguerre_on_dyadic_layer}, we get
    \begin{align}
        \sum\limits_{(k,l)\in K_3} |\mathcal{L}_n^{0}(X_{k,l}(a,b))|e^{-\frac{X_{k,l}(a,b)}{2}}
        & = \sum\limits_{j=0}^{M-1} \sum\limits_{(k,l)\in V_j}
        |\mathcal{L}_n^{0}(X_{k,l}(a,b))|e^{-\frac{X_{k,l}(a,b)}{2}}\\
        & \leq \widetilde{C} \sum\limits_{j=0}^{M-1} \frac{\#V_j}{U_j^{1/2}}\\
        & \leq \widetilde{C} \sum\limits_{j=0}^{M-1}  \frac{2^j n^{-1/3-\eta} + n^{1/6-\eta/2} + 2^{j/2} n^{\delta-\eta}}{2^{j/2}n^{1/6}}. 
    \end{align}
    We treat the latter sums separately:
    \begin{equation}
        \Sigma_1 := \sum\limits_{j=0}^{M-1} 2^{j/2}n^{-1/2-\eta}
        \leq \frac{2^{M/2}}{\sqrt{2}-1} n^{-1/2-\eta}
        \leq 4n^{p/2-2/3-\eta},
    \end{equation}
    \begin{equation}
        \Sigma_2 := \sum\limits_{j=0}^{M-1} 2^{-j/2}n^{-\eta/2}
        \leq 4n^{-\eta/2},
    \end{equation}
    \begin{equation}
        \Sigma_3 := \sum\limits_{j=0}^{M-1} n^{-1/6+\delta-\eta}
        \leq M n^{-1/6+\delta-\eta}
         \leq C n^{-1/6+\delta-\eta} \log_2 n.
    \end{equation}
    Since $p<1$ and  $0 < \delta < \eta$, we see that
    \begin{equation}
        \lim_{n \to \infty} \Sigma_\ell = 0, \quad \ell = 1,2,3.
    \end{equation}
Thus, 
\begin{equation}
     \lim\limits_{n \to \infty} \sup\limits_{(a,b) \in S_n(\delta, \eta)} \sum\limits_{(k,l)\in K} |\mathcal{L}_n^{0}(X_{k,l}(a,b))|e^{-\frac{X_{k,l}(a,b)}{2}} \leq \lim_{n \to \infty} \sup\limits_{(a,b) \in S_n(\delta, \eta)} (J_1 + \Sigma_1 +\Sigma_2+ \Sigma_3) = 0.
\end{equation}
This finishes the proof of Proposition~\ref{pro:Janssen_main}.
\end{proof}

\subsection{Proof of Theorem~\ref{thm:Janssen_test}}

It is easy now to finish the proof of the main theorem of this section.
By Proposition~\ref{pro:Janssen}, to verify that for sufficiently large $n$ the Gabor system $\mathcal{G}_n(a,b)$ forms a frame in $L^2(\R)$ for every $(a,b) \in S_n(\delta,\eta)$, it suffices to show that there exists $N\in\N$ such that the inequality
\begin{equation}\label{eq:vhh2}
    \sup\limits_{(a,b)\in S_n(\delta, \eta)} \sum_{(k,l)\in \Z^2} \left|V_{h_n} h_n \left( \frac{k}{b},\frac{l}{a} \right) \right| < 2 \quad \text{holds true for every $n>N$.}
\end{equation}

Using the relation between the short-time Fourier transform of $h_n$ in terms of the Laguerre polynomials as we did in~\eqref{eq:Janssen_sum} and $\mathcal{L}^0_n(0)=1$, we get 
\begin{equation}
 \begin{split}
    &\sum_{(k,l)\in \Z^2} \left|\mathcal L^0_n\left(X_{k,l}(a,b)\right)\right|
    e^{-X_{k,l}(a,b)/2} \\
    &\leq
    1+   \sum_{\substack{(k,l)\in\Z^2\\X_{k,l}(a,b)\geq 6n}}
         \left|\mathcal L^0_n\left(X_{k,l}(a,b)\right)\right| e^{-X_{k,l}(a,b)/2}
         +         \sum_{\substack{(k,l)\in K}}
         \left|\mathcal L^0_n\left(X_{k,l}(a,b)\right)\right| e^{-X_{k,l}(a,b)/2},
         \end{split}
\end{equation}
where $K$ is defined in~\eqref{eq:K_set_def}.
By Propositions~\ref{pro:Janssen_tail} and~\ref{pro:Janssen_main}, the sums on the right-hand side in the last estimate tend to $0$ as $n \to \infty$ uniformly in $a,b$.
Thus, \eqref{eq:vhh2} follows, and the theorem is proved.

\section{Far from the diagonal}

\begin{theorem}\label{thm:Wirtinger_test}
    For every $\varepsilon>0$ and $0<\rho<\frac{1}{2}$, there exists $n_{\varepsilon, \rho} \in \N$ such that, whenever $n\geq n_{\varepsilon, \rho}$ and a pair of parameters $a,b>0$ satisfies
    \begin{equation}\label{eq:ab_far}
        \min\{a,b\} \leq n^{-\frac{1}{2}- \varepsilon}, \quad ab \leq \frac{1}{2}-\rho,
    \end{equation}
     the Gabor system $\mathcal{G}_n(a,b)$ forms a frame for $L^2(\R)$.
\end{theorem}

\begin{proof}
        From the symmetry of the frame sets of Hermite functions, we may assume without loss of generality that
    \begin{equation}\label{eq:ab_far_a}
        a=\min\{a,b\} \leq n^{-\frac{1}{2}-\varepsilon}, \quad ab \leq \frac{1}{2}-\rho.
    \end{equation}

    To prove that the point $(a,b)$ belongs to the frame set, we apply the Wirtinger test in its rectangular form to the Gabor system $\mathcal{G}_n(a,b)$. By Corollary \ref{cor:Wirtinger_rectangular}, it suffices to show that
\begin{equation}\label{eq:ab_bound_to_prove}
    ab < \delta_{h_n}(b)= \frac{1}{2} \inf_{x\in[0,1]}
    \sqrt{\frac{\displaystyle\sum_{k\in\mathbb Z}|\widehat h_n(b(x-k))|^2}{\displaystyle\sum_{k\in\mathbb Z}(x-k)^2|\widehat h_n(b(x-k))|^2}}.
\end{equation}
Since $|\widehat h_n|^2$ is an even function for every $n\in \N$, we can take the infimum only over the interval $[0,\frac{1}{2}]$.

To treat simultaneously the cases of even and odd $n$, we choose $\tau\in\{0,1\}$ and $m\in \N_0$ such that
\begin{equation}
n=2m+\tau.
\end{equation}
We also set 
\begin{equation}
     P_{m,\tau}(x):= {\mathcal{L}}^{-\frac{1}{2}+\tau}_{m}(x).
\end{equation}
From \eqref{eq:h_n-L_n}, we obtain
$$
h_{2m+\tau}(t) = C_{m,\tau} t^{\tau} e^{-\pi t^2} P_{m,\tau}(2\pi t^2),
$$
where $C_{m,\tau}$ is a constant depending on $\tau$ and $m$, and we use the standard convention that $t^{\tau} = 1$ for every $t\in \R$ and $\tau=0$.
Using this and the eigenfunction property of Hermite functions under the Fourier transform, we rewrite \eqref{eq:ab_bound_to_prove} in terms of Laguerre polynomials:
\begin{equation}\label{eq:wirtinger_applied}
    4a^2 b^2 < \inf_{x\in[0,\frac{1}{2}]}
    \frac{\displaystyle\sum_{k\in\mathbb Z} (x-k)^{2\tau} e^{-2\pi b^2(x-k)^2}|P_{m,\tau}(2\pi b^2(x-k)^2)|^2}{\displaystyle\sum_{k\in\mathbb Z}(x-k)^{2\tau+2} e^{-2\pi b^2(x-k)^2}|P_{m,\tau}(2\pi b^2(x-k)^2)|^2}.
\end{equation}
Denote by 
$$
Q_k(x) := (x-k)^{2\tau}e^{-2\pi b^2(x-k)^2}|P_{m,\tau}(2\pi b^2(x-k)^2)|^2 \left(1-4a^2b^2 (x-k)^2\right)
$$
and
$$
F(x):=\sum_{k\in\mathbb Z} Q_k(x).
$$
Rewriting \eqref{eq:wirtinger_applied} in the new terms, we see that to prove the theorem, it suffices to verify
\begin{equation}\label{eq:F_positive}
F(x)> 0    \quad \text{for every }\, x\in\left[0,\frac{1}{2}\right].
\end{equation}

Let us briefly explain the strategy for proving this estimate.
It is clear that the sign of each non-zero term $Q_k$ in the sum above is determined solely by the sign of the factor
$1-4a^2b^2(x-k)^2.$
We split the sum into its positive and negative parts and estimate the positive part from below and the absolute value of the negative part from above.
One obstacle for obtaining a lower bound for the positive part is that the value of the Laguerre polynomial may be arbitrarily small when its argument is close to one of its zeros. To avoid this issue, we retain only one positive term, corresponding to an index $k_0$ chosen so that the argument of $P_{m,\tau}$ lies to the right of its largest zero. In this region, the product representation of $L_m$ provides a suitable lower bound.

Now, we pass to the proof of \eqref{eq:F_positive} and divide the argument into four steps.

\textbf{Step 1.}
Choose an index
\begin{equation}\label{eq:k_0_choice}
    k_0:=
    \left\lceil
        \frac{1}{b}\sqrt{\frac{5m}{6}}+\frac{1}{2}
    \right\rceil. 
\end{equation}
For a fixed $x\in[0,\frac{1}{2}]$, denote by
$$
\mathcal{N}_x:=\left\{k\in\Z:\ |x-k|>\frac{1}{2ab}\right\}
$$
the set of indices corresponding to the negative terms in the sum \eqref{eq:F_positive}.
\begin{claim}\label{claim:k_0_N_X}
   For all sufficiently large $m\in \N$ we have $k_0\notin \mathcal{N}_x$. 
\end{claim}
\begin{proof}[Proof of the Claim \ref{claim:k_0_N_X}]
Recall that \eqref{eq:ab_far_a} asserts that $a \sqrt{m} \leq 2^{-\frac{1}{2}-\varepsilon} m^{-\varepsilon} \to 0$ and $ab\leq \frac{1}{2}-\rho$. 
The claim follows immediately for $k_0=1$. If $k_0\geq2$, then one can check that the definition of $k_0$ gives
$$
b<\frac{\sqrt{5m}}{\sqrt{6}(k_0-\frac{3}{2})}.
$$
Using this inequality, we get
\begin{equation}\label{eq:2abk_0}
2ab|x-k_0|
\leq 2ab k_0 \leq
2 a\sqrt{\frac{5m}{6}} \frac{k_0}{k_0-\frac{3}{2}} \leq 8 a\sqrt{m} \leq 8 m^{-\varepsilon} \to 0, \quad \text{as } m \to +\infty,
\end{equation}
which proves the claim.
\end{proof}

We set
\begin{equation}\label{eq:c_rho}
    c_{\rho}:=2\rho(1-\rho),
\end{equation}
\begin{equation}\label{eq:S_1_def}
    S_1(x):=c_{\rho} (x-k_0)^{2\tau} e^{-2\pi b^2(x-k_0)^2}
\left|P_{m,\tau}\left(2\pi b^2(x-k_0)^2\right)\right|^2,
\end{equation}
and
\begin{equation}
    S_2(x):= 4a^2b^2 \sum_{k\in\mathcal{N}_x}  (x-k)^{2\tau+2}e^{-2\pi b^2(x-k)^2} \left|P_{m, \tau}\left(2\pi b^2(x-k)^2\right)\right|^2.
\end{equation}

It is clear that all terms in the definition of $F(x)$ corresponding to the indices $k\notin\mathcal{N}_x$ are non-negative.

\begin{claim}\label{claim:k_0_claim}
    The term $Q_{k_0}$ is bounded from below by $S_1(x)$ for sufficiently large $m$.
\end{claim}
\begin{proof}[Proof of the Claim \ref{claim:k_0_claim}]
We need to show that
\begin{equation}\label{eq:ab_c_rho}
1-4a^2b^2 (x-k_0)^2 \geq c_{\rho}    
\end{equation}
provided that $m$ is chosen large enough.

For $k_0=1$ the estimate simply follows from \eqref{eq:ab_far_a}:
$$
1-4a^2b^2(x-k_0)^2 \geq 1 - 4 a^2 b^2 \geq 1-4\left(\frac{1}{2} - \rho\right)^2 = 2 c_{\rho}. 
$$
For $k_0 \geq 2$ we have already established \eqref{eq:2abk_0}.
With this estimate in hand, we obtain
$$
1-4a^2b^2(x-k_0)^2 \geq 1 - (2 a b |x-k_0|)^2 \geq 1-64m^{-2\varepsilon}. 
$$
Therefore, \eqref{eq:ab_c_rho} follows for sufficiently large $m$.
\end{proof}

It is obvious that $S_2$ majorizes the absolute value of the negative part of the sum in \eqref{eq:F_positive}.
Together with this observation, Claim \ref{claim:k_0_claim} implies that 
\begin{equation}\label{eq:F_S1_S2}
    F(x)\geq S_1(x)-S_2(x), \qquad\text{for } x\in\left[0,\frac{1}{2}\right].
\end{equation}

\textbf{Step 2.} We proceed with estimating $S_1$ from below. 
Denote by $x_1 < \dots < x_m$ the zeros of polynomial $P_{m,\tau}$.
Observe that for $\tau \in \{0,1\}$ we have  $x_j \in (0,4m+4)$ for every $j=1,\dots,m.$ 
Note that for every $x\in [0,\frac{1}{2}]$ and sufficiently large $m$, the argument of the Laguerre polynomial in $S_1$ is outside of the oscillatory region:
\begin{equation}
   2\pi b^2 (k_0-x)^2 \geq 2\pi b^2 \left(k_0- \frac{1}{2}\right)^2 \geq \frac{5\pi m}{3} \geq 5m + 5. 
\end{equation}
Therefore, the value of $P_{m,\tau}$ can be estimated from below  using the product formula:
\begin{equation}\label{eq:Lm_lower_bound}
    |P_{m,\tau}(2\pi b^2 (k_0-x)^2)|=\frac{1}{m!}\prod\limits_{j=1}^m(2\pi b^2 (k_0-x)^2 - x_j) \geq \frac{m^m}{m!}\geq 1.
\end{equation}

Below, we treat the cases $k_0 = 1$ and $k_0 \geq 2$ separately. 

If $k_0 = 1$, then 
$2\pi b^2 (k_0-x)^2  \leq 2\pi b^2, \, (x-k_0)^{2\tau} \geq \frac{1}{4}$ for $x\in[0,\frac{1}{2}]$.
Together with \eqref{eq:Lm_lower_bound} this gives us 
\begin{equation}
    S_1(x) \geq \frac{c_{\rho}}{4}  e^{-2\pi b^2}.
\end{equation}
Now we deal with $k_0\geq 2$. Clearly, we still have $(x-k_0)^{2\tau}\geq\frac{1}{4}$ for $x\in[0,\frac{1}{2}]$. Furthermore, in this case, from \eqref{eq:k_0_choice} we have $b < 2\sqrt{\frac{5m}{6}}.$
Using this bound, for $x\in[0,\frac{1}{2}]$ we get
\begin{equation}\label{eq:k_0_upper_bound}
   2\pi b^2 (k_0-x)^2 \leq 2\pi b^2 k_0^2 \leq 2\pi b^2 \left(\frac{\sqrt{5m}}{b\sqrt{6}}+\frac{3}{2}\right)^2\leq \pi (5m+9 b^2) \leq 35\pi m, 
\end{equation}
which leads to $S_1(x) \geq \frac{c_{\rho}}{4} e^{-35 \pi m}.$

Putting both estimates together, we conclude that for sufficiently large $m$ the inequalities
\begin{equation}\label{eq:S_1_lower_bound}
    S_1(x)\geq
    \begin{cases}
         \frac{c_{\rho}}{4}\,e^{-2\pi b^2},
        & \text{if } k_0=1,\\[1ex]
         \frac{c_{\rho}}{4}\,e^{-35\pi m},
        & \text{if } k_0\geq2,
    \end{cases}
\end{equation}
hold true.

\textbf{Step 3.} In this step, we find an upper bound for $S_2$. 
We first verify that the arguments of Laguerre polynomials lie to the right of the oscillatory region. Since $k\in \mathcal{N}_x$, we have $|x-k|>\frac{1}{2ab}$, and \eqref{eq:ab_far_a} implies $a^2 \leq (2m)^{-1-2\varepsilon}$. Thus, for sufficiently large $m$ and every $x\in[0,\frac{1}{2}]$ we have
$$
2\pi b^2(x-k)^2 > \frac{\pi}{2a^2}\geq\frac{\pi}{2} (2m)^{1+2\varepsilon}>4m+4.
$$
This allows us to apply \eqref{eq:x_n_n_factorial_bound} and estimate the Laguerre polynomials by
$$
\left|P_{m,\tau}\left(2\pi b^2(x-k)^2\right)\right| \leq \frac{(2\pi)^m b^{2m}(x-k)^{2m}}{m!}.
$$
Therefore,
\begin{equation}\label{eq:S2_bound_1}
    S_2(x)\leq 4a^2b^2 \sum\limits_{k\in \mathcal{N}_x} \frac{(2\pi)^{2m} b^{4m}(x-k)^{4m+2\tau+2} e^{-2\pi b^2(x-k)^2}}{(m!)^2}.
\end{equation}
To find an upper bound for
$$
J:=\sum\limits_{k\in \mathcal{N}_x} (x-k)^{4m+2\tau+2} e^{-2\pi b^2(x-k)^2},
$$
our plan is to apply Lemma \ref{lem:tail_estimate} with $ r:=4m+2\tau+2,  M:=\frac{1}{2ab}$ and $\Phi(t):=2\pi b^2 t^2$. In order to do this, we need to verify that for sufficiently large $m$ the inequality
$$
\gamma := 2ab(4m+2\tau+2) -2\pi b^2\left(\frac{1}{ab}\right)<0
$$
holds true. Indeed, from \eqref{eq:ab_far_a}, we get $a^2 m \leq m^{-2\varepsilon}$ and $b/a \geq 1$, therefore 
$$
\gamma = \frac{b}{a} \left(2a^2 (4m+2\tau+2) - 2\pi\right) \leq 14 m^{-2\varepsilon} - 2\pi< -1
$$
provided that $m$ is chosen large enough.
The set $\mathcal{N}_x$ is the union of two shifted one-sided arithmetic progressions. Applying Lemma \ref{lem:tail_estimate} to each of them gives 
$$
J\leq \frac{2 e^{-\frac{\pi}{2a^2}}}{(2ab)^{4m+2\tau+2}(1-e^{\gamma})} \leq \frac{e^{-\frac{\pi}{2a^2}}}{2^{4m}(ab)^{4m+2\tau+2}}.
$$
Plugging this inequality into \eqref{eq:S2_bound_1}, we get the following upper bound for $S_2(x)$ for $x \in [0,\frac{1}{2}]$ and sufficiently large $m$:
\begin{equation}\label{eq:s2_bound_prelim}
S_2(x) \leq a^2 \pi^{2m} \frac{2^{2m+2}b^{4m+2}J}{(m!)^2}  \leq \frac{4}{(m!)^2} \left(\frac{\pi}{2a^2}\right)^{2m} (ab)^{-2\tau} e^{-\frac{\pi}{2a^2}} \leq  \left(\frac{\pi}{2a^2}\right)^{2m+2} e^{-\frac{\pi}{2a^2}}.  
\end{equation}
In the latter inequality, we used $(ab)^{-2\tau} \leq a^{-4\tau} \leq a^{-4}$ which holds since $b\geq a.$ 

Next, we split the exponential factor as 
$$e^{-\frac{\pi}{2a^2}}=e^{-\frac{\pi}{2a^2}(1-c_{\rho})}e^{-\frac{\pi}{2a^2}c_{\rho}},$$
where $c_{\rho}$ was defined in \eqref{eq:c_rho}.
\begin{claim}
    For sufficiently large $m$ the inequality
    \begin{equation}\label{eq:psi_m_negative}
    \left(\frac{\pi}{2a^2}\right)^{2m+2} e^{-\frac{\pi}{2a^2}c_{\rho}} \leq 1
\end{equation} is true.
\end{claim}
\begin{proof}
Note that the function $$\psi(t):=(2m+2)\log t- c_{\rho}t$$ is decreasing after the point $t_0=\frac{2m+2}{c_{\rho}}$. Furthermore, for sufficiently large $m$, we have $m^{1+2\varepsilon}>t_0$ and
$$
\psi(m^{1+2\varepsilon}) = (2m+2)(1+2\varepsilon) \log m - c_{\rho}m^{1+2\varepsilon} < 0, 
$$
whence $\psi$ is negative on $[m^{1+2\varepsilon},+\infty).$
Since $\frac{\pi}{2a^2} \geq \frac{\pi}{2}(2m)^{1+2\varepsilon}$, we obtain \eqref{eq:psi_m_negative}.   
\end{proof}

Combining \eqref{eq:s2_bound_prelim} and \eqref{eq:psi_m_negative}, we arrive at
\begin{equation}\label{eq:S2_final_estimate}
  S_2(x) \leq e^{-\frac{\pi}{2a^2}(1-c_{\rho})}  
\end{equation}
for sufficiently large $m$.

\textbf{Step 4.} In this step, we finish the proof of \eqref{eq:F_positive}. 
First, we deal with the case $k_0 = 1$. Then, the relations \eqref{eq:F_S1_S2}, \eqref{eq:S_1_lower_bound}, and \eqref{eq:S2_final_estimate} imply
$$
F(x) \geq S_1(x) - S_2(x) \geq \frac{c_{\rho}}{4}e^{-2\pi b^2}-e^{-\frac{\pi}{2a^2}(1-c_{\rho})}
$$
for every $x\in[0,\frac{1}{2}]$. From \eqref{eq:ab_far_a}, we get $4a^2b^2 \leq 1-2c_{\rho}.$
Hence, the difference  
$$
\frac{\pi}{2a^2}(1-c_\rho) - 2\pi b^2 = \frac{\pi}{2a^2}(1-c_\rho - 4 a^2 b^2) \geq  \frac{\pi c_{\rho}}{2a^2}
$$
tends to infinity as $m\to\infty$. This implies that $F$ is positive on $[0,\frac{1}{2}]$ for sufficiently large $m$.

Second, we assume $k_0\ge 2$. Then, applying again \eqref{eq:F_S1_S2}, \eqref{eq:S_1_lower_bound}, and \eqref{eq:S2_final_estimate}, we arrive at
$$
F(x) \geq S_1(x) - S_2(x) \geq \frac{c_{\rho}}{4}e^{-35\pi m}-e^{-\frac{\pi}{2a^2}(1-c_{\rho})}.
$$
Since $a^{-2}\geq (2m)^{1+2\varepsilon}$, the difference $\frac{\pi}{2a^2}(1-c_{\rho})-35\pi m \to +\infty$ as $m\to+\infty.$ Thus,  we have $F(x)> 0$ uniformly for $x\in[0,\frac{1}{2}]$ when $m$ is sufficiently large.

This finishes the proof of the theorem.
\end{proof}

\section{Proof of the Main Theorem}

The main theorem now easily follows from the results obtained in Secs.~3 and~4.

\begin{proof}[Proof of Theorem~\ref{thm:main}]
Fix $\eta_0,\delta_0>0$ and $\rho_0\in(0,\frac12).$
We start by showing that for sufficiently large $n$ the set
\begin{equation}
\mathcal{A}:=\left\{(a,b)\in\R_+^2:ab\leq n^{-\frac23-\eta_0} \right\}
\end{equation}
is contained in $\mathcal{F}(h_n).$
Put
$\eta:=\min\{\eta_0,\tfrac14\}, \,\delta:=\frac{\eta}{2}.
$
Note that $0<\delta<\eta<\frac{1}{3}$ and we can apply Theorem~\ref{thm:Janssen_test}. Since
\begin{equation}
   \mathcal{A}_1:=\left\{(a,b)\in\R_+^2:
ab\leq n^{-\frac23-\eta_0},\ 
\min\{a,b\}\geq n^{-\frac12-\delta}
\right\}
\subset S_n(\delta,\eta) ,
\end{equation}
we see that 
for all sufficiently large $n$ the region $\mathcal{A}_1$ 
is contained in the frame set of $h_n$.

On the other hand, since
$n^{-\frac23-\eta_0}\leq \frac12-\rho_0$
for all sufficiently large $n$, Theorem~\ref{thm:Wirtinger_test}, applied with
$\varepsilon=\delta$ and $\rho=\rho_0$, shows that
\begin{equation}
\mathcal{A}_2:=\left\{ (a,b)\in\R_+^2: ab\leq n^{-\frac{2}{3}-\eta_0},\  \min\{a,b\}< n^{-\frac{1}{2}-\delta} \right\}    
\end{equation}
is also contained in $\mathcal{F}(h_n)$.
Combining these two inclusions and using $\mathcal{A}=\mathcal{A}_1 \cup \mathcal{A}_2$, we obtain that $\mathcal{A}\subset\mathcal{F}(h_n).$

Finally, again applying Theorem~\ref{thm:Wirtinger_test} with
$\varepsilon=\delta_0$ and $\rho=\rho_0$, we deduce that
\begin{equation}
    \left\{
(a,b)\in\R_+^2:
ab\leq\frac12-\rho_0,\ 
\min\{a,b\}\leq n^{-\frac12-\delta_0}
\right\}
\end{equation}
is contained in the frame set for all sufficiently large $n$. Taking $n$ sufficiently large for all the above inclusions, we finish the proof of the main theorem.
\end{proof}

\section{Acknowledgments}
M.~Faulhuber was supported by the Austrian Science Fund (FWF) [\href{https://doi.org/10.55776/PAT5102224}{10.55776/PAT5102224}]. I.~Zlotnikov was supported by Grant 334466 of the Research Council of Norway and by ESI Research in Teams Grant "Gabor frames in higher dimensions".
I.~Shafkulovska was funded in part or in whole by the Austrian Science Fund (FWF)  [\href{https://doi.org/10.55776/Y1199}{10.55776/Y1199}].
For open access purposes, the authors have applied a CC BY public copyright license to any author accepted manuscript version arising from this submission.

\end{document}